\documentclass{amsart}
\usepackage{amssymb,amsmath}

\newtheorem{theorem}{Theorem}[section]
\newtheorem{lemma}[theorem]{Lemma}
\newtheorem{problem}[theorem]{Problem}

\newtheorem{corollary}[theorem]{Corollary}

\theoremstyle{definition}
\newtheorem{definition}[theorem]{Definition}

\theoremstyle{remark}

\numberwithin{equation}{section}

\newcommand{\bs}{\mathbf s}

\newcommand{\eps}{\varepsilon}

\def\Dio{{\rm Dio}}
\def\dio{{\rm dio}} 
\def\ice{{\rm ice}}
\def\rep{{\rm rep}}

\begin{document}

\title[On the expansions in two multiplicative dependent bases]{On the expansions of real 
numbers in two multiplicative dependent bases}  

\author{Yann Bugeaud}
\address{Department of Mathematics, Universit\'e de Strasbourg, 7 rue Ren\'e
Descartes, 67084 Strasbourg, France}
\email{bugeaud@math.unistra.fr}

\author{Dong Han Kim}
\address{Department of Mathematics Education,
Dongguk University -- Seoul, Seoul 04620, Korea.}
\email{kim2010@dongguk.edu}

\begin{abstract}
Let $r \ge 2$ and $s \ge 2$ be multiplicatively dependent integers.   
We establish a lower bound for the sum of the block complexities   
of the $r$-ary expansion and of the $s$-ary expansion 
of an irrational real number, viewed as infinite words on $\{0, 1, \ldots , r-1\}$
and $\{0, 1, \ldots , s-1\}$, and we show that this bound is best possible. 
\end{abstract}




\def\Dio{{\rm Dio}}
\def\dio{{\rm dio}} 
\def\ice{{\rm ice}}
\def\rep{{\rm rep}}
\def\Card{{\rm Card}}

\maketitle 

\section{Introduction}

Throughout this paper, $\lfloor x \rfloor$ denotes the greatest
integer less than or equal to $x$ 
and $\lceil x \rceil$
denotes the smallest integer greater than or equal to $x$.
Let $b \ge 2$ be an integer.
For a real number $\xi$, write
$$
\xi = \lfloor \xi \rfloor + \sum_{k \ge 1} \, {a_k \over b^k} = 
\lfloor \xi \rfloor + 0.a_1 a_2 \ldots ,
$$
where each digit $a_k$ is an integer from $\{0, 1, \ldots , b-1\}$ and
infinitely many digits $a_k$ are not equal to $b-1$.  
The sequence ${\bf a} := (a_k)_{k \ge 1}$ is uniquely determined by
the fractional part of $\xi$.
With a slight abuse of notation, we call it 
the $b$-ary expansion of $\xi$ and we view it also as the infinite word 
${\bf a} = a_1 a_2 \ldots $ over the alphabet $\{0, 1, \ldots , b-1\}$. 

For an infinite word ${\bf x} = x_1 x_2 \ldots $ over a finite alphabet 
and for a positive integer $n$, set
$$
p(n, {\bf x}) = \Card\{ x_{j+1} \ldots x_{j+n} : j \ge 0\}. 
$$
This notion from combinatorics on words is now commonly used to
measure the complexity of the $b$-ary expansion of a real number $\xi$. 
Indeed, for a positive integer $n$, we denote by 
$p(n, \xi, b)$ the total number of distinct blocks of $n$ digits
in the $b$-ary expansion ${\bf a}$ of $\xi$, that is,
$$
p(n, \xi, b) := p(n, {\bf a}) = \Card\{ a_{j+1} \ldots  a_{j+n} : j \ge 0\}. 
$$
Obviously, we have
$
1 \le p(n, \xi, b) \le b^n,
$
and both inequalities are sharp. 
If $\xi$ is rational, then its $b$-ary expansion is
ultimately periodic and the numbers $p(n, \xi, b)$, $n \ge 1$, 
are uniformly bounded by a constant depending only 
on $\xi$ and $b$. 
If $\xi$ is irrational, then, 
by a classical result of Morse and Hedlund \cite{MoHe},
we know that $p(n, \xi, b)\ge n+1$ for
every positive integer $n$, and this inequality is sharp.

\begin{definition} 
A Sturmian word $\mathbf x$ is an infinite word 
which satisfies
$$
p(n,{\mathbf x}) = n + 1, \quad \hbox{for $n \ge 1$}.
$$
A quasi-Sturmian word $\mathbf x$ is an infinite word 
which satisfies
$$
p(n,{\mathbf x}) = n + k, \quad \hbox{for $n \ge n_0$},
$$
for some positive integers $k$ and $n_0$. 
\end{definition}


The following rather general problem was investigated in \cite{Bu12}. 
Recall that two positive integers $x$ and $y$ are called 
{\it multiplicatively independent} if 
the only pair of integers $(m, n)$ such that
$x^m y^n = 1$ is the pair $(0, 0)$. 

\begin{problem}
Are there irrational real numbers having a `simple'
expansion in two multiplicatively independent bases?
\end{problem}



We established in \cite{BuKim15c} that  
the complexity function of the $r$-ary expansion of an 
irrational real number and that of its $s$-ary expansion cannot both grow too slowly when $r$ and $s$ are multiplicatively independent positive integers. 

\begin{theorem}[\cite{BuKim15c}]\label{twobases}
Let $r$ and $s$ be multiplicatively independent positive integers.
Any irrational real number $\xi$ satisfies
$$
\lim_{n \to + \infty} \, \bigl( p(n, \xi, r) + p(n, \xi , s) -  2n \bigr) = + \infty.
$$
Said differently, $\xi$ cannot have simultaneously a quasi-Sturmian 
$r$-ary expansion and a quasi-Sturmian $s$-ary expansion. 
\end{theorem}


We complement Theorem~\ref{twobases} by the following statement 
addressing expansions of a real number in two multiplicatively dependent bases. 


\begin{theorem}\label{twobasesdepter}    
Let $r, s \ge 2$ be multiplicatively dependent integers and $m, \ell$ be the smallest 
positive integers such that $r^m = s^\ell$.
Then, there exist uncountably many real numbers $\xi$ satisfying  
$$
\lim_{n \to + \infty} \, \bigl( p(n, \xi, r) + p(n, \xi , s) -  2n \bigr) = m + \ell  
$$ 
and every irrational real number $\xi$ satisfies   
$$
\lim_{n \to + \infty} \, \bigl( p(n, \xi, r) + p(n, \xi , s) -  2n \bigr) \ge m + \ell.  
$$ 
\end{theorem}

The next result, used in the proof of Theorem \ref{twobasesdepter}, has its own interest.

\begin{theorem}\label{twobasesdepbis}
Let $b \ge 2$ be an integer and $\rho, \sigma$ be positive integers.
If $\sigma$ divides $\rho$, then every real number whose $b^\rho$-ary expansion is
quasi-Sturmian has a quasi-Sturmian $b^\sigma$-ary expansion.  
Moreover, every real number whose $b^\rho$-ary and $b^\sigma$-ary expansions 
are both quasi-Sturmian has a quasi-Sturmian $b^\mu$-ary expansion, 
where $\mu$ is the least common multiple of $\rho$ and $\sigma$.    
\end{theorem}


We conclude by an immediate consequence of Theorems \ref{twobases} 
and \ref{twobasesdepter}. 

\begin{corollary}   
Let $r, s \ge 2$ be distinct integers.  
No real number can have simultaneously a Sturmian $r$-ary expansion 
and a Sturmian $s$-ary expansion. 
\end{corollary} 

Our paper is organized as follows. 
Section 2 gathers auxiliary results on Sturmian and quasi-Sturmian words.
Theorems \ref{twobasesdepter} and \ref{twobasesdepbis} are established in Section 4.


\section{Auxiliary results}

Here and below, for integers $i, j$ with $i \le j$, we write 
$x_i^j$ for the factor $x_i x_{i+1} \ldots x_j$ of $\mathbf x$.

We will make use of the following characterisation of quasi-Sturmian words.

\begin{lemma}\label{Cas}
An infinite word ${\mathbf x}$ written over a finite alphabet ${\mathcal A}$ 
is quasi-Sturmian if and only if there are a 
finite word $W$, a Sturmian word $\bs$ defined over $\{0, 1\}$ and a 
morphism $\phi$ from $\{0, 1\}^*$ into ${\mathcal A}^*$ such that 
$\phi (01) \not= \phi (10)$ and 
$$
{\mathbf x} = W \phi (\bs).
$$
\end{lemma}

\begin{proof}
See \cite{Cassa98}.
\end{proof}

Throughout this paper, for a finite word $W$ and an integer $t$, we write $W^t$ 
for the concatenation of $t$ copies of $W$ and $W^{\infty}$ 
for the concatenation of infinitely many copies of $W$. 
We denote by $|W|$ the length of $W$, that is, the number of letters composing $W$. 
A word $U$ is called periodic if $U = W^t$ for some finite word $W$ and an integer $t \ge 2$.
If $U$ is periodic, then the period of $U$ is defined as the length of the shortest 
word $W$ for which there exists an integer $t \ge 2$ such that $U = W^t$. 

\begin{lemma}\label{periodic}  
Let $U$ be a finite word. Assume that there exist words $U_1, U_2, V, W$ such that  
$U = U_1 U_2$ and  $U U = V U_2 U_1 W$, 
with $|U_1| \ne |V|$ and $0 < |V| < |U|$.
Then, the word $U$ is periodic.  
\end{lemma}  

\begin{proof}
Since $V$ is a prefix of $U$ and $W$ is a suffix of $U$,    
we get 
$$
U = U_1 U_2 = VW,
$$    
thus, $VU_2U_1W = U U = VWVW$.   
This implies 
$$
U_2U_1 =  WV.
$$    
If $|U_1| < |V|$, then we can write $V = V' U_1$ for a nonempty word $V'$,
thus $U_2 = W V'$.
Therefore,
$$
U_1 W V' = U_1 U_2 = VW = V' U_1 W.
$$
Our assumption $0 < |V| < |U|$ implies that the word $Z:= U_1 W$ is nonempty. 
Since $Z V' = V' Z$, it follows from  
Theorem~1.5.3 of \cite{AlSh03} that $U = Z V' $ is periodic. 
The proof of the case $|U_1| > |V|$ is similar.    
\end{proof}

\begin{lemma}\label{qsrep}
Let ${\mathcal A}$ be a finite set, 
$\bs$ a Sturmian word over $\{0, 1\}$, and $\phi$ a morphism  
from $\{0, 1\}^*$ into ${\mathcal A}^*$ satisfying $\phi(01) \ne \phi(10)$.  
Then there exists an integer $n_0$ such that, for any factor $A$ of $\bs$ 
of length greater than $n_0$, 
if one can write $\phi(A)$ as $V_1 \phi(b_2 b_3 \dots b_{m-1} ) V_2$, where   
$B = b_1b_2  \dots b_{m-1} b_m$ is a factor of $\bs$, the word   
$V_1$ is a nonempty suffix of $\phi(b_1)$, and $V_2$   
is a nonempty prefix of $\phi(b_m)$,  
then $V_1 = \phi(b_1), V_2 = \phi(b_m)$   and $A = B$.    
\end{lemma}

\begin{proof}
We may assume that $1$ is the isolated letter in $\bs$, i.e., that 
$11$ is not a factor of $\bs$.  
Since $\bs$ is balanced, there exists a positive integer $k$ 
such that $1 0^t 1$ is a factor of $\bs$ if and only if $t = k$ or $k+1$.

We first consider the case where $V_1 = \phi(b_1)$.   
Suppose that $A \ne B$.     
Then, by deleting the maximal common prefix of $A$ and $B$,     
we may assume that $A$ and $B$ have no common prefix.   
Thus, the prefixes of $A$ and $B$ are $00$ and $10$.   

If $\phi(00) = \phi(10) V_2$, then $\phi(0) = \phi(1) V_2 = V_2 \phi(1)$ and there exist 
a word $U$ and positive integers $s, t$ such that 
$\phi(1) = U^s$ and $\phi(0) = U^t$. This gives a contradiction to $\phi(01) = \phi(10)$.

If $\phi(10) = \phi(0^h) V_2$ for some integer $h \ge 2$ and a nonempty prefix 
$V_2$ of $\phi(0)$, then, writing $\phi(0) = V_2 V'$, we get $\phi(0) = V_2 V' = V' V_2$, 
thus there exist a word $U$ and positive integers $s, t$ such that 
$\phi(1) = U^s$ and $\phi(0) = U^t$. This gives a contradiction to $\phi(01) = \phi(10)$.

If $\phi(10) = \phi(0^h) V_2$ for some integer $h \ge 2$ and a nonempty prefix $V_2$ 
of $\phi(1)$, then there exists a positive integer $\ell$ and a prefix $V'$ of $\phi(0)$ such that 
$\phi(1) = \phi(0)^\ell V'$. 
Write $\phi(0) = V' V''$. Then, $\phi(10) = \phi(0)^\ell V' \phi(0) = \phi(0)^{\ell+1} V'$
and we get $V' \phi(0) = \phi(0) V'$. 
Thus, there exist a word $U$ and positive integers $s, t$ such that 
$\phi(1) = U^s$ and $\phi(0) = U^t$. This gives a contradiction to $\phi(01) = \phi(10)$.

Similarly, we show that, if $V_2 = \phi(b_m)$, then $A = B$. 

It only remains for us to treat the case where $V_1 \ne \phi(b_1)$ and $V_2 \ne \phi(b_m)$.   
There exists an integer $n_0$ such that any factor $A$ of $\bs$     
of length greater than $n_0$ contains $10^k10^{k+1}10$.    
It is sufficient to consider the case where   
$\phi(10^k10^{k+1}10) = V_1 \phi(b_2 b_3 \dots b_{m-1}) V_2$, 
for a factor $b_1 b_2 \dots b_m$ of $\bs$  and 
with $V_1$ a proper nonempty suffix of $\phi(b_1)$      
and $V_2$ a proper nonempty prefix of $\phi(b_m)$.      

If $b_2 b_3 \dots b_{m-1} = 0^{k+1}10^{k}1$, then $b_1 = 1$ and $b_m = 0$.    
Thus $|V_1| < |\phi(1)|$ and $|V_2| < |\phi(0)|$,     
which contradicts    
$$
|V_1| + |V_2| <   |\phi(1)| + |\phi(0)| = |\phi(10^k10^{k+1}10)| - |\phi(0^{k+1}10^{k}1)|.    
$$  
Therefore,  since any subword of $\bs$  
in which $10^k10$ and $10^{k+1}1$ do not occur 
is a factor of $0^{k+1}10^{k}1$,  we deduce that 
if $\phi(10^k10^{k+1}10) = V_1 \phi(b_2 \dots b_{m-1}) V_2$ as above,    
then $b_2 \dots b_{m-1}$ contains $10^k10$ or $10^{k+1}1$.

We distinguish three cases:

Case (i) : $\phi(10^k10^{k+1}10) = W_1 \phi(10^k10) W_2$, 
where $0 < |W_1| < |\phi(10^{k})|$. 
\\ \noindent
Then 
$$
\phi(10^k10^{k}) = W_1 \phi(10^k) W'_2, 
\qquad  \phi(0^k1 0 0^{k}10) = W'_1 \phi(0^k10) W_2, 
$$
where $|W'_2| = |W_2| - |\phi(0)|$ and $|W'_1| = |W_1|$.

Case (ii) : $\phi(10^k10^{k+1}10) = W_1 \phi(10^k10) W_2$, 
where $|\phi(10^{k})| < |W_1| < |\phi(10^{k+1})|$. 
\\ \noindent
Then 
$$
\phi(10^k10^{k}) = W'_1 \phi(0^k1) W'_2,  
\qquad  \phi(0^k10 0^{k}10) = W''_1 \phi(0^k10) W_2,   
$$
where $|W'_1| = |W_1| - |\phi(0^{k})|$, $|W'_2| = |W_2| + |\phi(0^{k-1})|$ and $|W''_1| = |W_1|$.

Case (iii) : $\phi(10^k10^{k+1}10) = W_1 \phi(10^{k+1}1) W_2$, 
where $0 < |W_1| < |\phi(10^{k+1})|$. 
\\ \noindent
Then 
$$
\phi(10^k10^{k}) = W_1 \phi(10^k) W'_2,  
\qquad  \phi(0^k1 0 0^{k}10) = W'_1 \phi(0^{k+1}1) W_2, 
$$
where $|W'_2| = |W_2| - |\phi(0)|$ and $|W'_1| = |W_1|$.

By Lemma~\ref{periodic}, in each Case (i), (ii), (iii), 
the factors  $\phi(10^k)$ and $\phi(0^k10)$ are periodic. 
Denoting by $\lambda_1, \lambda_2$ the periods of $\phi(10^k)$, $\phi(0^k10)$, we get  
$$
\lambda_1 \le \frac{|\phi(10^k)|}{2} = \frac{k |\phi(0)| + |\phi(1)|}{2},   
\quad \lambda_2 \le \frac{|\phi(0^k10)|}{2} = \frac{(k+1) |\phi(0)| + |\phi(1)|}{2}.   
$$
Write $\phi(10^k) = U^t$ for a word $U$ with $|U| = \lambda_1$ and integer $t \ge 2$. 
Then $\phi(1) = U^{t_1} U_1$,  $\phi(0^k) = U_2 U^{t_2}$ 
for some words $U_1, U_2$ with $U= U_1 U_2$   
 and some nonnegative integers $t_1, t_2$ satisfying $t_1 + t_2 = t -1$.    
Thus, we get 
$$
\phi(0^k1) = U_2 (U_1 U_2)^{t_2} (U_1 U_2)^{t_1} U_1 = (U_2 U_1)^t,   
\qquad |U_2U_1| = \lambda_1. 
$$   
Since $\phi(0)$ is a prefix of $(U_2 U_1)^t$,  
we deduce that 
$\phi(0^k10) = (U_2 U_1) \cdots (U_2 U_1) U'$ for a prefix $U'$ of $U_2 U_1$,  
It then follows from 
\cite[Lemma 3 (v)]{CK}  that $\lambda_1 = \lambda_2$ or 
$$
|\phi(0^k10)|< \lambda_1 + \lambda_2 \le (k+\frac 12) |\phi(0)| + |\phi(1)| < |\phi(0^k10)|,     
$$ 
in which case we have a contradiction.
If  $\lambda_1 = \lambda_2$, then $\lambda_1$ divides $|\phi(0^k10)|$ and $|\phi(10^k)|$, 
thus $\lambda_1$ divides $|\phi(0)|$ and $|\phi(1)|$.   
This implies that  $\phi(01) = \phi(10) = U U \cdots U$, giving again a contradiction.   
\end{proof}

We end this section with an easy result on the convergents of irrational numbers.

\begin{lemma}\label{lemnew}
Let $(\frac{p_k}{q_k})_{k \ge 0}$ be the sequence of convergents of an irrational number 
$[0; a_1, a_2, \ldots]$
in $(0, 1)$ and $d \ge 2$ be an integer.
Let $c_1$, $c_2$ be integers not both multiple of $d$.  
Then, for any positive integer $k$, we have
$c_1 p_k + c_2 q_k \not\equiv 0 \pmod d$ 
or $c_1 p_{k+1} + c_2 q_{k+1} \not\equiv 0 \pmod d$.
\end{lemma}

\begin{proof}
Since 
$$ 
\begin{bmatrix} p_{k} & p_{k+1} \\ q_{k} & q_{k+1} \end{bmatrix} 
= \begin{bmatrix} 0 & 1 
\\ 1 & a_{1} \end{bmatrix} \begin{bmatrix} 0 & 1
 \\ 1 & a_{2} \end{bmatrix} \cdots \begin{bmatrix} 0 & 1
  \\ 1 & a_{k+1} \end{bmatrix},
  $$
we have 
$$ 
\begin{bmatrix} c_1 p_k + c_2 q_k & c_1 p_{k+1} + c_2 q_{k+1} \end{bmatrix} = 
\begin{bmatrix} c_1 & c_2 \end{bmatrix}  \begin{bmatrix} 0 & 1 
\\ 1 & a_{1} \end{bmatrix} \begin{bmatrix} 0 & 1 \\ 
1 & a_{2} \end{bmatrix} \cdots  \begin{bmatrix} 0 & 1 \\ 1 & a_{k+1} \end{bmatrix},
$$
thus 
$$ 
\begin{bmatrix} c_1 & c_2 \end{bmatrix}  =  
\begin{bmatrix} c_1 p_k + c_2 q_k & c_1 p_{k+1} + c_2 q_{k+1} \end{bmatrix}
\begin{bmatrix} -a_{k+1} & 1 \\ 1 & 0 \end{bmatrix}\cdots \begin{bmatrix} -a_{2} & 1 \\ 
1 & 0 \end{bmatrix} 
\begin{bmatrix} -a_{1} & 1 \\ 1 &  0 \end{bmatrix}.
$$
Hence, if $\begin{bmatrix} c_1 p_k + c_2 q_k & c_1 p_{k+1} + c_2 q_{k+1} \end{bmatrix} 
= \begin{bmatrix} 0 & 0 \end{bmatrix}$ modulo $d$,
then $c_1$ and $c_2$ are multiple of $d$. 
\end{proof}


\section{Proofs of Theorems~\ref{twobasesdepter} and~\ref{twobasesdepbis}}

We begin with the proof of Theorem~\ref{twobasesdepbis}.   

\begin{proof}[Proof of Theorem~\ref{twobasesdepbis}]
Let $b \ge 2$ be an integer and $\rho, \sigma$ be positive integers.
Assume that $\rho = d \sigma$ for some integer $d \ge 2$. 
Let $\xi$ be a real number and assume that there are integers 
$a_1, a_2, \ldots $ in $\{0, 1, \ldots , b^\rho-1\}$ and $k$, $n_0$ such that
$$
\xi = \lfloor \xi \rfloor + \sum_{i\ge 1} \frac{a_i}{b^{\rho i}} \ \text{ and } \ p(n, \xi, b^\rho) =  n + k 
\text{  for } n \ge n_0.
$$
Then, by Lemma~\ref{Cas}, there are a finite word $W$, a Sturmian word $\bs$ 
defined over $\{0,1\}$ and a morphism $\phi$ from $\{ 0, 1\}^*$ into $\{ 0, 1, \dots, b^\rho-1 \}^*$ 
such that $\phi(01) \ne \phi(10)$ and 
$$
{\mathbf a} = a_1 a_2 \ldots = W \phi( \bs).
$$
Let $a$ be in $\{0, 1, \ldots , b^\rho -1\}$ and consider its representation in 
base $b^\sigma$ given by 
$a = c_1 b^{(d-1)\sigma} + c_2 b^{(d-2)\sigma} + \ldots + c_d b^{0\cdot \sigma}$, where 
$c_1, \ldots , c_d$ are in $\{0, 1, \ldots , b^\sigma -1\}$. 
Define the function $\phi_{\rho,\sigma}$ on $\{0, 1, \ldots , b^\rho -1\}$ 
by setting $\phi_{\rho,\sigma} (a) = c_1 c_2 \ldots c_d$.  
It extends to a morphism 
from $\{0, 1, \ldots , b^\rho -1\}^*$ to $\{0, 1, \ldots , b^\sigma -1\}^*$, 
which we also denote by $\phi_{\rho,\sigma}$. 
Then, we have
$$
\xi = \lfloor \xi \rfloor + \sum_{i\ge 1} \frac{d_i}{b^{\sigma i}},  \ \text{ where } \ {\mathbf d}
 = d_1 d_2 \ldots = \phi_{\rho, \sigma}(W) \, (\phi_{\rho,\sigma} \circ \phi) ( \bs). 
$$
We deduce from Lemma \ref{Cas} that the $b^\sigma$-ary expansion of $\xi$ is quasi-Sturmian.  
Thus we have established the first assertion of the theorem.

For the second assertion of the theorem, we may assume that $\rho$ and $\sigma$ 
are relatively prime (otherwise, we replace $b$ by $b^g$  
where $g$ is the greatest common divisor of $\rho$ and $\sigma$).

Let $\xi$ be a real number and write
$$
\xi = \lfloor \xi \rfloor + \sum_{i\ge 1} \frac{a_i}{b^{\rho i}} 
=  \lfloor \xi \rfloor + \sum_{j \ge 1} \frac{b_j}{b^{\sigma j}}, 
$$
where $a_1, a_2, \ldots $ are in $\{0, 1, \ldots , b^\rho-1\}$ and $b_1, b_2, \ldots $ 
are in $\{0, 1, \ldots , b^\sigma-1\}$.
Assume that ${\mathbf a} = a_1 a_2 \ldots$  and ${\mathbf b} = b_1 b_2 \ldots$ 
are both quasi-Sturmian.
By Lemma~\ref{Cas}, there are a finite word $W$, a Sturmian word $\bs$ 
defined over $\{0,1\}$ and a morphism $\phi$ from $\{ 0, 1\}^*$ into $\{ 0, 1, \dots, b^\rho-1 \}^*$ 
such that $\phi(01) \ne \phi(10)$ and 
$$
{\mathbf a} = a_1 a_2 \ldots = W \phi( \bs).
$$

We claim that $|\phi(0)| =: l_0 $ and $|\phi(1)|=: l_1$ are both multiple of $\sigma$. 

In order to deduce a contradiction, we suppose that $\sigma$ 
does not divide at least one of $l_0$ and $l_1$.

Let $\phi_{\rho,1}$ be the morphism $\phi_{\rho,\sigma}$ defined above in the case $\sigma = 1$. 
For each factor $U$ of $\bs$, let 
$$
\Lambda (U) := \{  0 \le j \le \sigma -1 :   
\phi_{\rho,1} ({\mathbf a}) 
= V \phi_{\rho,1} \circ \phi (U) \text{ for some $V$ with } |V| \equiv j \pmod \sigma \}
$$
denote the nonempty set of positions modulo $\sigma$ where $\phi_{\rho,1} \circ \phi(U)$ 
occurs in $\phi_{\rho,1} (\mathbf a)$.
If $U'$ is a prefix of $U$, then $\Lambda(U)$ is a subset of $\Lambda(U')$.
Consequently, there exists $N$ 
such that $\Lambda( s_1 \dots s_n) = \Lambda( s_1 \dots s_N) $ for each $n \ge N$.

Let $[0; a_1, a_2, \ldots ]$ denote the continued fraction expansion of the 
slope of $\bs$ and, for $k \ge 1$, let $q_k$ be the denominator of the convergent
$[0; a_1, \ldots , a_k]$ to this slope. 
Define the sequence $(M_k)_{k \ge 0}$ of finite words over $\{0, 1\}$ by  
$$
M_0 = 0, \quad M_1 = 0^{a_1 -1} 1, \quad  
\hbox{and} \quad M_{k+1} = (M_k)^{a_k} M_{k-1}, \quad (k \ge 1).
$$
For $k \ge 1$, the word $M_k$ is a factor of length $q_k$ of $\bs$ (see e.g. \cite{Loth02}).
Since there are $p_k$ occurrences of the digit $1$ in $M_k$, we get    
$$
|\phi(M_k)| = l_0 (q_k - p_k) + l_1 p_k = (l_1 - l_0) p_k + l_0 q_k.
$$
By Lemma~\ref{lemnew} and the assumption that $\sigma$ 
does not divide at least one of $l_0$ and $l_1$, 
we conclude that at least one of $|\phi(M_k)|$ and $|\phi(M_{k+1})|$ is not a multiple of $\sigma$.

Let $U$ be a factor of $\bs$. 
Then $U$ is a factor of $M_k$ for some integer $k$.  
Since $M_k M_k$ is a factor of 
$M_{k+2}M_{k+1} = (M_{k+1})^{a_{k+2}} M_k (M_k)^{a_{k+1}} M_{k-1}$, 
which is a factor of $\bs$, there are two positions of $\phi(U)$
which differ by $|\phi(M_k)|$. 
Thus, there exist two occurrences of 
$\phi(U)$ in $\phi(\bs)$ separated by exactly $\rho |\phi(M_k)|$ letters.
Replacing $k$ by $k+1$ is necessary,
we can assume that $\rho  |\phi(M_k)|$ is not a multiple of $\sigma$
and we deduce that $|\Lambda(U)| \ge 2$ for any factor $U$ of $\bs$.

A finite word $U$ is called right special if $U$ is a prefix of two different factors   
of $\bs$ of the same length. 
If the initial word $s_1 \dots s_n$ of $\bs$ is not a prefix of a right special word, then either   
$s_{j+1} \dots s_{j+n} \ne s_1 \dots s_n$ for all $j \ge 1$, or $\bs$ is periodic. 
Since a Sturmian word is recurrent and not periodic (see, e.g., \cite[page 158]{Fogg02}),  
there are infinitely many prefixes $s_1 \dots s_n$ of $\bs$ which are right special. 
Let $n \ge N$ be such that $s_1 \dots s_n$ is right special. 
Then, there exists a letter $c$ such that  $c \ne s_{n+1}$ and $s_1 \dots s_n c$ is a factor of $\bs$.
Thus, we get  
$$
\Lambda(s_1 \dots s_n s_{n+1}) = \Lambda(s_1 \dots s_n ) \supset \Lambda(s_1 \dots s_n c ). 
$$ 
Choose $i, j$ in $\Lambda(s_1 \dots s_n c ) $ with $0 \le i < j \le \sigma-1$. 
Then we can write 
$$
\phi_{\rho,1}({\mathbf a}) = U U_1 \phi_{\rho,1} \circ \phi (s_1 \dots s_n c ) U'_1 \ldots
= U' U_2 \phi_{\rho,1} \circ \phi (s_1 \dots s_n s_{n+1})  U'_2 \ldots 
$$ 
and 
$$
\phi_{\rho,1}({\mathbf a}) = V V_1 \phi_{\rho,1} \circ \phi(s_1 \dots s_n c ) V'_1 \ldots 
= V' V_2 \phi_{\rho,1} \circ \phi(s_1 \dots s_n s_{n+1}) V'_2 \ldots, 
$$
for some words $U, U', V, V', U_1, U_2, V_1, V_2, U'_1, U'_2, V'_1, V'_2$ 
written over $\{0, \ldots , b-1\}$ and satisfying
$$
|U_1| = |U_2| = i, \ |V_1| = |V_2| =j, \ |U| \equiv |U'|\equiv |V|\equiv |V'| \equiv 0 \pmod \sigma, 
$$
$$
0 \le |U'_1| =  |U'_2| \le \sigma-1, \quad 0 \le  |V'_1| = |V'_2| \le \sigma-1, 
$$
and $\sigma$ divides $i+(n+1)\rho + |U'_1|$ and $j + (n+1)\rho + |V'_1|$.
Thus, there exist $u_1, u_2, v_1, v_2$ in $\{0, 1, \ldots , b^\sigma -1\}$ 
and words $X, Y, A_1, A_2, B_1, B_2$ written over $\{0, 1, \ldots , b^\sigma -1\}$  
with 
$$
|X| =\Bigl\lfloor \frac{i+n\rho}{\sigma} \Bigr\rfloor -1, \quad |Y| 
= \Bigl\lfloor \frac{j + n \rho}{\sigma} \Bigr\rfloor -1
$$ 
and 
$$
A_1 \ne A_2,  \quad B_1 \ne B_2,  \quad |A_1| = |A_2| < \frac \rho \sigma +2,  
\quad |B_1| = |B_2| < \frac{\rho}{\sigma} +2,
$$
such that 
\begin{align*}
U_1 \phi_{\rho,1} \circ \phi (s_1 \dots s_n c ) U'_1 &=  \phi_{\sigma,1}(u_1 X A_1), \\
U_2 \phi_{\rho,1}\circ \phi(s_1 \dots s_n s_{n+1}) U'_2 &=  \phi_{\sigma,1}(u_2 X A_2), \\
V_1 \phi_{\rho,1} \circ\phi(s_1 \dots s_n c ) V'_1 &=   \phi_{\sigma,1}(v_1 Y B_1), \\
V_2 \phi_{\rho,1}\circ \phi (s_1 \dots s_n s_{n+1}) V'_2 &=  \phi_{\sigma,1}(v_2 Y B_2).    
\end{align*} 
Here, $\phi_{\sigma,1}$ is defined analogously as $\phi_{\rho,1}$.
Therefore, $u_1 X A_1$, $u_2 X A_2$ and $v_1 Y B_1$, $v_2 Y B_2$ 
are all factors of $\phi_{\sigma,1}^{-1} ( \phi_{\rho,1}( \phi(\bs)))$. 
Denoting by $A$ (resp., by $B$) 
the longest common prefix (it could be the empty word) of $A_1$ and $A_2$ 
(resp., of $B_1$ and $B_2$), we deduce that $XA$ and $YB$ are both right special. 

Let $W_0$ be the longest common prefix of $\phi_{\rho,1} \circ \phi(s_1 \dots s_n s_{n+1})$ 
and $\phi_{\rho,1} \circ \phi(s_1 \dots s_n c)$.   
Then, there exist finite words $W_1, W_2, W'_1, W'_2$ over $\{0, \ldots , b-1\}$
satisfying $|W_1| = \sigma-i$, $|W_2| = \sigma-j$, $|W'_1| < \sigma$, $|W'_2| <\sigma$, and  
\begin{equation*}
W_0 = W_1 \phi_{\sigma,1}(XA) W'_1 = W_2\phi_{\sigma,1}(YB) W'_2,
\end{equation*}   
Thus, we get $|XA| \le |YB| \le |XA| +1$. 

Suppose that $XA$ is a suffix of $YB$.   
Then, there exists a nonempty finite word $W'$ of length less than $\sigma$ such that  
\begin{align*}
W_0 &= W_2 W' \phi_{\sigma,1}(XA) W'_1 = W_2 \phi_{\sigma,1}(XA) W'_2, &\text{ if } |XA| = |YB|, \\  
W_0 &= W_1 \phi_{\sigma,1}(XA) W'_1 = W_1 W' \phi_{\sigma,1}(XA) W'_2, &\text{ if } |XA| +1 = |YB|.   
\end{align*}
It then follows from Theorem 1.5.2 of \cite{AlSh03} that we have  
$W_0 = W_2 (W')^t W'' W'_1$ or $W_1 (W')^t W'' W'_2$, respectively, 
for some integer $t$ and a prefix $W''$ of $W'$.  
Since $\rho, \sigma$ are fixed and $\bs$ is Sturmian, 
we deduce from Lemma 2.3 of \cite{BuKim15c} that $(W')^t$    
cannot be a factor of $\phi_{\rho,1} \circ \phi (s_1 \dots s_n )$ when $n$ is sufficiently large. 
This shows that the lengths of $XA$ and $YB$ are bounded independently of $n$. 

Consequently, the right special words $XA$ and $YB$ are not suffixes of each others if $n$
is sufficiently large. 
Hence, there are arbitrarily large integers $m$ such that 
$\phi_{\sigma,1}^{-1} \circ \phi_{\rho,1} \circ \phi (\bs)$ 
has two distinct right special words of length $m$.  
This implies that ${\mathbf b} = \phi_{\sigma,1}^{-1} \circ \phi_{\rho,1} ({\mathbf a})$ 
is not quasi-Sturmian, which gives a contradiction.
Therefore, we have established that $|\phi(0)|$ and $|\phi(1)|$ are both multiple of $\sigma$.  

Write
$$ 
\xi = \lfloor \xi \rfloor + \sum_{i\ge 1} \frac{c_i}{b^{\rho \sigma i}}, 
\qquad {\mathbf c} = c_1 c_2 \ldots 
= \phi_{\rho \sigma, \rho}^{-1} ({\mathbf a}) = \phi_{\rho \sigma, \rho}^{-1} (W \phi(\bs)) .
$$   
Put $|W| = h \sigma + d $ for integers $h \ge 0$ and $d$ with $0 \le d < \sigma$.
Let $\phi(0) = X_1 X_2$, $\phi(1) = Y_1 Y_2$, where $|X_1| = |Y_1| = \sigma - d$.   
Assume that $11$ is not a factor of $\bs$.   
Then there exists a positive integer $k$ such that $10^m1$ is a factor of $\bs$ 
if and only if $m = k$ or $k+1$.  
Thus, we can represent $\bs$ as    
$$ 
\bs = 0^{w} t_0 t_1 t_2 t_3 \ldots , \qquad t_0 = 10^k, \ t_i \in \{ 10^k, 0\}, \ 0 \le w \le k+1. 
$$    
It is not difficult to check that $\mathbf t := t_0 t_1 t_2 \ldots$ is Sturmian.
Define  $\phi'$ by
$$
\phi' (10^k) = X_2 Y_1 Y_2 (X_1 X_2)^{k-1} X_1, \quad \phi'(0) = X_2 X_1.
$$   
Then we get 
$$
\phi(\bs) = (X_1 X_2)^{w} Y_1 Y_2 (X_1 X_2)^{k-1} X_1 \phi'(t_1t_2 t_3 \dots),     
$$
thus  
$$
{\mathbf c}= \phi_{\rho \sigma, \rho}^{-1}(W \phi(\bs))
=\phi_{\rho \sigma, \rho}^{-1}(W (X_1 X_2)^{w}Y_1 Y_2 (X_1 X_2)^{k-1} X_1)
(\phi_{\rho \sigma, \rho}^{-1} \circ \phi') (t_1t_2 t_3 \dots).
$$
Since $|\phi(0)|$ and $|\phi(1)|$ are both multiple of $\sigma$, the
morphism $\phi_{\rho \sigma, \rho}^{-1} \circ \phi'$ is well-defined. 
We conclude that $\mathbf c$ is quasi-Sturmian and the proof of the theorem is complete. 
\end{proof}

\begin{lemma}\label{lem4.2}
Let $b \ge 2$, $d \ge 2$, $\rho$, $\sigma$ be positive integers with $\rho = d \sigma$.
Let $x_1 x_2 \ldots $ be a quasi-Sturmian word over $\{ 0, 1, \ldots, b^\rho -1 \}$. 
Then, there exists an integer $n_0$ such that the real number 
$\xi = \sum_{k \ge 1} \frac{x_k}{b^{\rho k}}$ satisfies 
$$
p(nd, \xi , b^\sigma ) \ge (n +1)d,  \quad \text{ for $n\ge n_0$}.    
$$ 
Furthermore, if $s_1 s_2 \ldots $ is a Sturmian word written over $\{ 0, 1 \}$,  then
there exists an integer $n_0$ such that the real number 
$\xi = \sum_{k \ge 1} \frac{s_k}{b^{\rho k}}$ satisfies
$$
p(n, \xi , b^\sigma )  = n + d,  \quad \text{ for $n\ge n_0$}.    
$$ 
\end{lemma}

\begin{proof}
Set $\mathcal A := \{ 0, 1, \ldots, b^\rho -1 \}$. There exist a Sturmian word $\bs$
written over $\{0, 1\}$, a morphism $\phi$ from $\{0,1\}^*$ into $\mathcal A^*$
satisfying $\phi(01) \not= \phi(10)$, and a factor $W$ of  
$\mathbf x := x_1 x_2 \ldots$ such that ${\mathbf x} = W \phi(\bs)$.   
Then, the word 
$$
{\mathbf y} :=  \phi_{\rho,\sigma} ({\mathbf x}) 
= \phi_{\rho,\sigma}( W \phi(\bs)) =  \phi_{\rho,\sigma}( W) (\phi_{\rho,\sigma} \circ \phi) (\bs) 
$$    
is quasi-Sturmian.    

Let $n$ be a positive integer larger than the integer $n_0$ given by  
Lemma~\ref{qsrep} applied to the morphism $\phi_{\rho,\sigma} \circ \phi$.   
We claim that if $U_1 \phi_{\rho,\sigma} (A_1) V_1 = U_2 \phi_{\rho,\sigma} (A_2) V_2$, 
where $A_1, A_2$ are factors of $\phi(\bs)$ of length $n$ and $U_1, U_2$ (resp., $V_1, V_2$) are
nonempty suffixes (resp., proper prefixes) 
of words of the form $\phi_{\rho,\sigma} (a)$ for $a$ in $\mathcal A$,
then $U_1 = U_2$, $A_1 = A_2$ and $V_1 = V_2$.

Suppose not. Then we may assume that there exist $A_1, A_2$ and $U, V$ such that
$$
\phi_{\rho,\sigma} (A_1) V = U \phi_{\rho,\sigma} (A_2).
$$
Thus there exist $a_1, a_2$ in $\mathcal A$, 
a factor $A$ of $\phi(\bs)$ of length $n$, and a factor 
$A'$ of $\phi(\bs)$ of length $n-1$ such that 
$\phi_{\rho,\sigma} (A) = W_1 \phi_{\rho,\sigma} (A') W_2$, 
where $W_1$ (resp., $W_2$) is a nonempty proper suffix (resp., prefix)
of $\phi_{\rho,\sigma} (a_1)$ (resp., of $\phi_{\rho,\sigma} (a_2)$).  
Consequently, there exist $b, b', c, c'$ in $\{0, 1\}$ and factors  
$B, B'$ of $\bs$ such that  $A = U \phi(B) V$,  $a_1 A' a_2 = U' \phi(B') V'$,     
where $U$ (resp., $U'$) is a nonempty suffix of $\phi(b)$ (resp., $\phi(b')$) and 
$V$ (resp., $V'$) is a nonempty prefix of $\phi(c)$ (resp., $\phi(c')$).  
Then $A' = U'' \phi(B') V''$ for words $U'', V''$ such that $U' = a_1 U''$, $V' = V'' a_2$. 
Therefore, we get    
$$
\phi_{\rho,\sigma} (A) = \phi_{\rho,\sigma} (U) (\phi_{\rho,\sigma} \circ \phi) (B) \phi_{\rho,\sigma} (V)    
= W_1 \phi_{\rho,\sigma} (U'') (\phi_{\rho,\sigma} \circ \phi) (B') \phi_{\rho,\sigma} (V'') W_2.   
$$
We deduce from Lemma~\ref{qsrep} that  
$\phi_{\rho,\sigma} (U) = W_1 \phi_{\rho,\sigma} (U'')$,  
$\phi_{\rho,\sigma} (V)= \phi_{\rho,\sigma} (V'') W_2$   
and $B = B'$.    
This is a contradiction to the fact that  
$W_1$ (resp., $W_2$) is a nonempty proper suffix (resp., prefix) 
of $\phi_{\rho,\sigma} (a_1)$ (resp., of $\phi_{\rho,\sigma} (a_2)$).   
Hence, the representation of $X = U \phi_{\rho,\sigma} (A) V$ is unique.   

If $\phi(\bs)$ is written over an alphabet of three letters or more, then   
$$
p (n-1, \phi (\bs)) \ge (n-1) +2 = n+1,
$$   
which implies that the number of factors $X$ of $(\phi_{\rho,\sigma} \circ \phi) (\bs)$
of length $nd$ is at least equal to $(n+1)d$.    
If $\phi(\bs)$ is written over an alphabet of two letters, say over the alphabet $\mathcal A = \{a,b\}$,   
then we can put $\phi_{\rho,\sigma} (a) = Z X$ and $\phi_{\rho,\sigma} (b) =  Z Y$, 
where $Z$ is the longest common prefix of 
$\phi_{\rho,\sigma} (a), \phi_{\rho,\sigma} (b)$ and the first letters of $X, Y$ are different.    
If $|V| > |Z|$, then for each right special factor $A$ of $\bs$     
there are two distinct factors $\phi_{\rho,\sigma}(A)V_1$, 
$\phi_{\rho,\sigma}(A)V_2$ in $\phi(\bs)$.  
If $|V| \le |Z|$, then $|U| \ge |X| = |Y|$, thus for each left special factor $B$ of $\bs$   
there are two factors $U_1\phi_{\rho,\sigma}(B)$, $U_2\phi_{\rho,\sigma}(B)$ in $\phi(\bs)$. 
For each $c = 0, \ldots , d-1$, the number of factors $X = U \phi_{\rho,\sigma} (A) V$ 
of $(\phi_{\rho,\sigma} \circ \phi) (\bs)$ of length $nd$ with $|A| = n-1$ and
$|U| = d - |V| =  c$ is at least equal to $p (n-1, \phi (\bs))+1$.      
Therefore, we get    
$$
p(nd, \xi , b^\sigma ) \ge p(nd, (\phi_{\rho,\sigma} \circ \phi) (\bs)) \ge (n +1) d.     
$$   
Since the function $m \mapsto p(m, \xi, b^\sigma)$ is strictly increasing,  
this implies the first assertion of the theorem.  

For the second assertion, 
let $\bs = s_1 s_2 \ldots $ be a Sturmian word written 
over the subset $\{0,1\}$ of $\{0, 1, \ldots , b^{\rho} -1\}$ and define 
$$
\xi = \sum_{i \ge 1} \frac{s_i}{b^{\rho i}}.
$$
Since $\phi_{\rho, \sigma} (0) = 0^d$ and $\phi_{\rho,\sigma} (1) = 0^{d-1}1$,
for $n \ge 1$, any factor of length $d n$ of $\phi_{\rho,\sigma} (\bs)$ is a suffix of
$\phi_{\rho,\sigma} (A) 0^k$, where $A$ is a factor of length $n$ in $\bs$ and  $0 \le k \le d-1$.     
Since $0^{d-1}$ is a prefix of $\phi_{\rho,\sigma} (A) 0^k$, 
the number of suffixes of $\phi_{\rho,\sigma} (A) 0^k$ of length $nd$ is $d(n+1)$, thus
$$
p(d n, \xi, b^\sigma ) = d(n+1) = dn + d.  
$$
Since the function $m \mapsto p(m, \xi, b^\sigma)$ is strictly increasing, 
this completes the proof of the theorem. 
\end{proof}

\begin{proof}[Proof of Theorem~\ref{twobasesdepter}]
Suppose that the two bases $r \ge 2$ and $s \ge 2$ are multiplicatively dependent and 
let $m, \ell$ be the coprime positive integers satisfying 
$r^m = s^\ell$. 
Then, there exists a positive integer $b$ such that $r = b^\ell$ and $s = b^m$.  

Let $\bs = s_1 s_2 \ldots $ be a Sturmian word 
over the subset $\{0,1\}$ of $\{0, 1, \ldots , b^{m\ell} -1\}$ and define 
$$
\xi = \sum_{i \ge 1} \frac{s_i}{b^{m\ell i}}.
$$
By the second assertion of Lemma~\ref{lem4.2}, 
there exists an integer $n_0$ such that 
$$
p(n, \xi , b^\ell )  = n + m  \quad \text{ and } \quad
p(n, \xi , b^m )  = n + \ell,  \quad \text{ for $n \ge n_0$}.    
$$ 
Thus,
$$
\lim_{n \to + \infty} \, \bigl( p(n, \xi, r) + p(n, \xi , s) -  2n \bigr) = m + \ell.
$$ 
This proves the first assertion of the theorem. 

For the second assertion of the theorem, 
it is sufficient to consider a real number $\xi$ whose $b^\ell$-ary and $b^m$-ary expansions 
are both quasi-Sturmian.
By Theorem~\ref{twobasesdepbis}, 
the $b^{\ell m}$-ary expansion of $\xi$ is also quasi-Sturmian and 
we deduce from the first assertion of Lemma~\ref{lem4.2} that 
there exists an integer $n_0$ such that 
$$
p(m n, \xi , b^\ell ) \ge m (n+1) \quad \text{ and } \quad
p(\ell n, \xi , b^m ) \ge \ell (n+1),  \quad \text{ for $n \ge n_0$}.     
$$ 
Therefore,
$$
\lim_{n \to + \infty} \, \bigl( p(n, \xi, r) + p(n, \xi , s) -  2n \bigr) \ge m + \ell.
$$ 
This completes the proof of the theorem.  
\end{proof}

\medskip
\section*{Acknowledgement}
Dong Han Kim was supported by the 
National Research Foundation of Korea (NRF-2015R1A2A2A01007090).


\end{document}